\documentclass[12pt,reqno]{amsart}

\usepackage{amsmath}%
\usepackage{amsfonts}%
\usepackage{amssymb}%
\usepackage{graphicx}
\usepackage{srcltx}%
\usepackage{dsfont}
\usepackage[latin1]{inputenc}
\usepackage{graphicx}
\usepackage{color}

\newtheorem{theorem}{Theorem}
\theoremstyle{plain}

\newtheorem{lemma}{Lemma}[section]

\newtheorem{proposition}{Proposition}[section]
\newtheorem{remark}{Remark}[section]

\numberwithin{equation}{section}

\newcommand{\E}{\mathbb{E}}
\newcommand{\N}{\mathbb{N}}
\newcommand{\R}{\mathbb{R}}

\newcommand{\pp}{\mathbb{P}}
\newcommand{\V}{{\mathcal{V}}}

\newcommand{\qw}{\begin{equation}}
\newcommand{\qwe}{\end{equation}}
\newcommand{\qww}{\begin{equation*}}
\newcommand{\qwee}{\end{equation*}}

\newcommand{\zn}{{\zeta_n(T)}}

\newcommand{\inc}[1]{\mathds{1}\left\{{#1}\right\}}

\newcommand{\bn}{\beta_n}
\newcommand{\bnn}{\beta_n^2}
\newcommand{\sqn}{\sqrt{n}}
\newcommand{\pin}{\sqrt{2\pi n}}

\newcommand{\tem}{T_{\epsilon}^{M}(n)}
\newcommand{\te}{T^{\epsilon}(n)}
\newcommand{\tm}{T_M(n)}

\newcommand{\ep}{\epsilon}

\newcommand{\an}{{\alpha_n}}

\newcommand{\vn}{{\mathcal{V}_n}}
\newcommand{\en}{\mathcal{E}_n}
\newcommand{\taun}{{\tau_x}}
\newcommand{\btaun}{\boldsymbol{\tau}_{n}}
\newcommand{\btau}{\boldsymbol{\tau}}
\newcommand{\yn}[1]{{Y_n({#1})}}
\newcommand{\yi}{Y_n(i)}
\newcommand{\rr}{\mathcal{R}}

\newcommand{\hideownbib}[1]{} 

\hideownbib{\newcites{own}{[a] Refereed}}
\hideownbib{\bibliographystyleown{plain}}

\begin{document}
\title[Extremal Aging For Trap Models]{Extremal Aging For Trap Models}
\author{Onur G\"{u}n}

\address
{Onur G\"{u}n, Weierstrass Institute Berlin\newline
\indent 39 Mohrenstrasse\newline
\indent 10117 Berlin, GERMANY}%
\email{guen@wias-berlin.de}%

\date{4 December 2013}%

\begin{abstract}
In the seminal work \cite{BC08}, Ben Arous and \v{C}ern\'y give a general characterization of aging for trap models in terms of $\alpha$-stable subordinators with $\alpha \in (0,1)$. ~Some of the important examples that fall into this universality class are Random Hopping Time (RHT) dynamics of Random Energy Model (REM) and $p$-spin models observed on exponential time scales. In this paper, we explain a different aging mechanism in terms of {\it extremal processes} that can be seen as the extension of $\alpha$-stable aging to the case $\alpha=0$. We apply this mechanism to the RHT dynamics of the REM for a wide range of temperature and time scales. The other examples that exhibit extremal aging include the Sherrington Kirkpatrick (SK) model and $p$-spin models \cite{BG12,BGS13}, and biased random walk on critical Galton-Watson trees conditioned to survive \cite{CFK13}.
\end{abstract}
\maketitle

\section{Introduction}

A striking feature common to many disordered systems is that in an out-of-equilibrium phase convergence to equilibrium is very slow, and it gets slower as the system gets `older'. This phenomena is usually called {\it aging} and has been studied extensively in physics community, both experimentally and theoretically. {\it Spin glasses} constitute an important class of disordered systems where aging occurs, and the dynamics of mean field spin glasses is one of the focal points of this article.

\newcommand{\G}{\mathcal{G}}

On the theoretical side, {\it trap models} have become the central paradigm to study aging properties of disordered systems. 
They are simple enough to track analytically and many examples exhibit aging that can be established in a mathematically 
rigorous way. The general description of trap models is the following. Let $\G=(\V,\mathcal{E})$ be a countable, connected graph where 
$\V$ and $\mathcal{E}$ denote the set of vertices and edges, respectively. 
Consider a random collection of positive numbers $\btau=\{\tau_x:x\in \V\}$. A {\it trap model}, $(X(t):t\geq 0)$, is a 
continuous time Markov process on $\V$ whose dynamics is described as follows: at a vertex $x\in\V$, it waits for an exponential time 
with mean $\tau_x$ and then moves to one of the neighbors of $x$ uniform at random. 
Here, one can view a vertex $x$ as a {\it trap} and $\tau_x$ as the {\it depth} of the trap $x$. Let $(Y(i):i\in\N_0)$ denote the discrete time simple random walk on $\V$.
Another view of trap models is that $X$ follows $Y$ that gets trapped at vertices, collecting exponential random variables, and hence, is a time change of it. 
This time-change process $S$, called the {\it clock-process}, is defined for $k\in\N$ by
\begin{equation}
S(k)=\sum_{i=0}^{k-1}\tau_{Y(i)}e_i,\;\;\;t\geq 0
\end{equation}
where $(e_i:\;i\in\N_0)$ is a collection of i.i.d. exponential mean-one random variables. The clock process and the Markov chain $Y$ completely describes $X$ via
\begin{equation}
	X(t)=Y(S^{\leftarrow}(t))      
\end{equation}
where $S^{\leftarrow}$ is the generalized right inverse of $S$. 

One central question in the study of the trap models is the analysis of the clock process, more precisely,
 the influence of the random trapping landscape. 
If the trapping landscape is very heterogenous on certain large times, one expects that the main contribution to the clock process come from few `deep traps'.
Moreover, if the graph is transient enough on these time scales, the walk, upon leaving a deep trap, will get throughly lost before finding a new deep trap, 
and as a result, the jumps of the clock process will be i.i.d. with the completely annealed distribution as the common distribution.  
Summation of random variables being dominated by few large ones signifies a stable convergence for the clock process. 
In the situations just described, the properly rescaled clock process converges to an $\alpha$-stable subordinator with $0<\alpha<1$. 
From this, one can easily deduce that the system exhibits aging, namely, the two-time correlation function
\begin{equation}\label{twotime}
R(t_1,t_2):=\pp(X_n(t_1)=X_n(t_2)|\tau)
\end{equation}
can be approximated, for large $t_1$ and $t_2$ corresponding time scales of observation, 
by the probability that an $\alpha$-stable subordinator jumps over the interval $[t_1,t_2]$. The latter is a function of the ratio $t_1/t_2$ and is given by the classical arcsine law. The existence of a correlation function that
depends only on the trajectory of $X$ on a time interval $[t_1,t_2]$ whose large time limit is a non-trivial function of the ratio $t_1/t_2$
is usually described as aging both in mathematics and physics literature.

The picture described in the above paragraph has been made rigorous in \cite{BC08} in the context of sequence of diverging finite graphs.
In this article, too, we will study a growing sequence of finite graphs. The reason to study trap models on finite graphs is twofold: Firstly, it 
allows one to employ potential theoretical methods strongly. Secondly, our main motivation, the mean field spin glasses, is in this setup.

Our main goal is to understand the one end of the spectrum of the $\alpha$-stable aging, namely, the case $\alpha=0$. 
More precisely, we are interested in the situations where as the graphs grow, the heavy-tail index $\alpha$ converges to 0.  
In this case, heterogeneity becomes even stronger in the sense that the main contribution to the clock process comes from the `deepest' of 
the deep traps the walk can find. Hence, the limiting clock process has a structure of a `record' process, namely, it is an {\it extremal process}. More precisely, after a linear rescaling, contributions from deep traps still grow on an exponential scale. As a result, firstly, the maximal term dominates, and secondly, one has to perform a further non-linear scaling, cohorent with case of sum of i.i.d. random variables with slowly varying tails (see e.g. \cite{D52, KS86}).

In the spirit of \cite{BC08} we will give a set of conditions on the trapping landscape and the potential theoretical properties of the graph which will ensure that the clock process can be approximated, for appropriate large time scales, by an extremal process. Next, we will describe two additional conditions that guarantees that the two-time correlation function in (\ref{twotime}) can be approximated using extremal processes. Here, in order to get a non-trivial limit, one has to slightly enlarge the ratio of the two times with the volume, due to the non-linear scaling of clock process. We have called this type of aging {\it extremal aging}.

Let us now discuss the trap model dynamics of spin glass models, also known as Random Hopping Time (RHT) dynamics. Our focus is on mean field models. The simplest mean field spin glass model is Random Energy Model (REM) introduced by Derrida in \cite{Der1}, where there are no correlation between energies. The first aging results for REM were proved in \cite{BBG03a} and \cite{BBG03b} for time scales only slightly shorter than equilibration times, using renewal techniques. Later, \cite{BC08} proved that the dynamics of REM model ages with the arcsine law as the limiting aging function. The time scales where REM exhibits aging were later extended in \cite{CG08}. In the second part of this article we extend these results to include extremal aging. We will let the temperature vary with the volume of the system in order to get a richer picture and prove that the dynamics of REM exhibits extremal aging for a wide range of temperature and time scales. More precisely, extremal aging occurs for fixed positive temperatures and subexponential time scales (in dimension of the system); for temperatures vanishing with the volume and for exponential, subexponential and superexponential time scales; and, for temperatures diverging with the volume and for (very) subexponential time scales. These results also signify that the aging is related to how transient the relaxation of the system is when far from equilibrium and it might have little to do with the actual equilibrium properties of the system. The occurrence of aging even in infinite temperature is a strong demonstration of this fact.

Let us mention the results on correlated mean field spin glass models. The arcsine law as an aging scheme, surprisingly, proved to be true even for some correlated spin glasses. In \cite{BBC08}, authors proved that the same aging mechanism essentially holds true for short enough time scales for the $p$-spin models. Later, finer results on the aging of $p$-spin models were achieved in \cite{BG13}. 

Our original motivation to study the $\alpha=0$ case, or rather the dynamics of mean field spin glasses on subexponential time scales, stemmed from the aim of extending the REM universality to the dynamics of the Sherrington Kirkpatrick (SK) model. The results on the statistical properties of the SK model (see \cite{BGK08}) indicated that in order to not feel the correlations in the model, one has to investigate the dynamics on subexponential time scales. The dynamics on such time scales are in the category of extremal aging. In \cite{BG12} we proved that on subexponential time scales the clock process of the SK model and $p$-spin SK models converge to extremal processes and these systems exhibit extremal aging. Finer analysis was carried out later in \cite{BGS13}. 

Recently, extremal aging has been observed for a type of model different than spin glass models. In \cite{CFK13}, it was proved that biased random walk on critical Galton-Watson trees conditioned to survive exhibits extremal aging.

The rest of the article is organized as follows: In Section \ref{scClock} we describe precisely the models 
we study, give a set of conditions on the trapping landscape and potential theoretical 
properties of the graph, and prove that these conditions imply the convergence of the rescaled clock process to an extremal process. 
In Section \ref{sc3} we set two additional conditions that leads to extremal aging. In Section \ref{sc4} we prove our results on the dynamics of REM.

\section{Convergence of the clock process to an extremal process.}\label{scClock}

We start this section by introducing precisely the type of trap models we will study. Let $\mathcal{G}_n=(\vn,\en)$, $n\in\N$,
 be a sequence of finite, connected graphs with $\vn$ and $\en$ denoting the set of vertices and edges, respectively. 
We use the notation $x\sim y$ for an edge $(x,y)\in\en$. For each $n\in\N$ and vertex $x\in\vn$ we attach a positive number $\tau_x$ 
which constitutes the "depth" of the trap at site $x$. We denote the collection of depths by $\btaun=\{\taun:\;x\in\vn\}$. We will 
call $\btaun$ the `trapping landscape' or `environment' and later we will choose $\btaun$ random. Given the trapping landscape $\btaun$ we define a continuous time Markov process $\{X_n(t):t\geq 0\}$ on $\vn$ whose transition rates are given by
$w_{xy}^{\btaun}=\inc{x\sim y}/(d_x \tau_x)$ where $d_x=\#\{y\in\vn:x\sim y\}$ is the degree of the vertex $x$. 
In other words, at a vertex $x$, the Markov process $X_n$ waits an exponential time with mean $\taun$ and than it moves to one 
of its neighbors chosen uniformly at random. We denote by $\pp_x$ and $\E_x$ the distribution and expectation of $X_n$ starting from 
$x\in\vn$. We will always start $X_n$ from an arbitrary but fixed vertex we denote by $0$ that does not depend on $\btaun$ and 
write for simplicity $\pp=\pp_{0}$ and $\E=\E_0$. Note that, although $X_n$ depends on $\btaun$ we surpressed it in the notation.

For each $n\in\N$ we take the trapping landscapes $\btaun\in (0,\infty)^{\vn}$ random with $P_n$ and $E_n$ 
denoting its distribution and expectation, respectively. We embed all the random variables $\btaun$, $n\in\N$, 
independently into a common probability space and $P$ and $E$ stands for the distribution and expectation of this probability space. 
For the events that happens $P$ almost surely, we will simply say `for a.s. random environment $\btau$'. 

For any $n\in\N$, let $\{Y_n(i):i\in\N_0\}$ be the simple random walk on the graph $\mathcal{G}_n$, that is,
 the discrete time Markov chain on $\vn$ with the transition probabilities $p_{xy}=\inc{x\sim y}/d_x$ and we 
set the starting point $Y_n(0)=0$. We assume that the distribution of $Y_n$ is defined in the probability space $\pp_0$.
  We define the clock process $S_n$ by setting $S_n(u)=0$ for $u\in[0,1)$ and for $u\in[1,\infty)$
\begin{equation}
S_n(u):=\sum_{i=0}^{\lfloor u\rfloor -1}\tau_{\yi}e_i
\end{equation}
where $\{e_i:i\in\N\}$ is an i.i.d. collection of exponential mean-one random variables independent from anything else. In words, for $k\in\N$, $S_n(k)$ is the time it takes for $X_n$ to make $k$ jumps. Clearly we have,
\begin{equation}
X_n(t)=Y_n(k) \;\;\text{ if   }S_n(k)\leq t < S_n(k+1).
\end{equation}

For $\btaun$ fixed, $S_n$ is a random variable taking values in $D([0,\infty))$, the space of c\'adl\'ag function on $[0,\infty)$. 
Similarly, for any fixed $T>0$ the restriction of $S_n$ to $[0,T]$ is a random variable on $D([0,T])$,  the space of c\'adl\'ag function on $[0,T]$.

We need now further notation. Let $T_n$ be a stopping time for the chain $Y_n$. We define $G_{T_n}^n(x,y),x,y\in\vn$ to be the Green's function of $Y_n$, that is, the mean number of times $Y_n$ visits $y$ before $T_n$, started from $x$:
\begin{equation}
 G_{T_n}^n(x,y):= \E_x\left[\;\sum_{i=0}^{T_n-1}\inc{Y_n(i)=y}\right].
\end{equation}
For $A\subseteq \vn$ we define the first hitting time of $A$ by
\begin{equation}
H_n(A):=\inf\{i\geq 0:\;Y_n(i)\in A\}.
\end{equation}
For ease of notation we write $G_A^n$ for $G_{H_n(A)}^n$. Finally, we say that a random subset $A\subseteq \V_n$ is a percolation cloud with density $\rho\in(0,1)$ if $x\in A$ with probability $\rho$, independently of all other vertices. 

Presently, we set 4 conditions that are about the trapping landscape and the properties of the walk $Y_n$ on $\V_n$. This set of very general, potential theoretical conditions will be used to prove our main result.

The first condition tells that a certain density of traps exceed a depth scale in a very heterogeneous way.

\newcommand{\gn}{g_n}
\newcommand{\rhon}{b_n}
\noindent{\bf Condition A:} For any $n\in\N$ let $\btaun$ be i.i.d. in $x\in\V_n$. There exists a {\it depth rate scale} $\alpha_n$, a {\it depth scale} $\gn$ and a {\it density scale} $\rhon$ with $\gn\to\infty$ and $\an,\rhon\to 0$ as $n\to\infty$ such that
\begin{equation}\label{condA1}
\rhon^{-1}P_n\left(\tau_x\geq u^{1/\an}\gn\right)\overset{n\to\infty}\longrightarrow {1}/{u}
\end{equation}
uniformly in $u$ on all compact subsets of $(0,\infty)$. Moreover, there exists a constant $C$ such that for all $u>0$ and $d>0$, for all $n\in\N$
\begin{equation}\label{condA2}
P_n\left(\tau_x\geq u^{1/\an}d\gn\right) \leq \frac{C\rhon}{ud^{\an}}.
\end{equation}

\vspace{0.12in}

The next two conditions concern the potential theoretical properties of the graph.

\newcommand{\fn}{f_n}
\newcommand{\tn}{t_n}

\noindent{\bf Condition B:} Let $\rhon$ be as in Condition A. Let $A_n,n\in\N$ be a sequence of percolation clouds 
on $\vn$ with densities $\rho\rhon$ where $\rho\in(0,\infty)$. There exists a scale $\fn$ with $\fn\to\infty$ as $n\to\infty$ and a constant $\mathcal{K}_G$ independent of $\rho$ such that for a.s. sequence $A_n$ 
\begin{equation}
\max_{x\in A_n}\left|\fn^{-1}G_{A_n\setminus \{x\}}^n(x,x)-\mathcal{K}_G\right|\overset{n\to\infty}\longrightarrow 0.
\end{equation}

\newcommand{\rn}{r_n}

\noindent{\bf Condition C:} Let $A_n$ be as in Condition B. There exists a constant $\mathcal{K}_r>0$ such that for all $s>0$ and a.s. sequence $A_n$
\begin{equation}
\max_{x\in A_n\cup \{0\}}\left|\E_x\left[\exp(-\frac{s}{\rn} H_n(A_n\setminus \{x\}))\right]-\frac{\mathcal{K}_r\rho}{s+\mathcal{K}_r \rho}\right|\overset{n\to\infty}\longrightarrow 0,
\end{equation}
where $\rn=\fn/\rhon$.
\vspace{0.12in}

The last condition contains technical restrictions.

\noindent{\bf Condition D:} 
\begin{itemize}
\item[(i)] There exists a sequence of positive numbers $\lambda_n$ and a positive constant $\mathcal{K}_s$ such that for all $T>0$ and $n$ large enough
\begin{equation}
\sum_{x\in\V_n}\big(e^{\lambda_n G_{T\rn}(0,x)}-1\big)\leq \mathcal{K}_s \lambda_n T\rn
\end{equation}
and $\sum_{n=1}^\infty \exp(-c \lambda_n \fn)<\infty$ for some $c>0$.

\vspace{0.12in}

\item[(ii)] $\an\log(\fn)\longrightarrow 0$ as $n\to\infty$.
\end{itemize}

\vspace{0.12in}

Now we introduce formally the extremal processes. Let $F$ be a probability distribution function on $(-\infty,\infty)$. 
For $l\in\N$ and $0\leq t_1\leq\cdots\leq t_l$ define the finite dimensional distributions
\begin{equation}\label{fddext}
F_{t_1,\dots,t_l}(x_1,\dots,x_l)=F^{t_1}\left(\bigwedge_{i=1}^l x_i\right) F^{t_2-t_1}\left(\bigwedge_{i=2}^l x_i\right)\cdots F^{t_l-t_{l-1}}\left(\bigwedge_{i=l}^l x_l\right)
\end{equation}
where $\bigwedge$ stands for the minimum. The family of finite dimensional distributions defined by (\ref{fddext}) is consistent and thus, by Kolmogorov extension theorem there exists a continuous time stochastic process $W=(W(t):t\geq 0)$ with these finite dimensional distributions. $W$ is called an extremal process generated by $F$ or $F$-extremal process.

We give another description of an extremal process. Let $F$ be as in the previous paragraph. Assume that $F$ is a continuous distribution with $supp(F)=\R$. Let $N$ denote a Poisson Random Measure (PRM) on  $[0,\infty)\times(0,\infty)$ with mean measure $dt\times \nu(dx)$ where $\nu(x,\infty)=-\log F(x)$. Let us denote by $(t_j,\xi_j)$ the marks of $N$. Then if we define
\begin{equation}
W(t)=\max_{t_i\leq t} \xi_i,\;\;\;t\geq 0  
\end{equation}
$W$ is an $F$-extremal process. It is enough to check that $W$ satisfies (\ref{fddext}) for any $0\leq t_1\leq\cdots\leq t_l$ and for 
any $x_1\leq x_2\leq \cdots\leq x_l$. By the continuity of $F$ and the independence properties of a PRM we get
\begin{equation}\label{poschr}
\begin{aligned}
&P(W(t_1)\leq x_1,\dots W(t_l)\leq x_l)\\&=P\big(N([0,t_1]\times [x_1,\infty))=0\big)\cdots P\big(N([t_{l-1},t_l]\times [x_l,\infty))=0\big)
\\&= e^{-t_1 \nu(x_1,\infty)}\cdots e^{-(t_l-t_{l-1})\nu(x_l,\infty)}
\\&= F^{t_1}(x_1)\cdots F^{t_l-t_{l-1}}(x_l). 
\end{aligned}
\end{equation}
For details about extremal processes we refer readers to \cite{Res}.

For our convergence results of clock processes we use two different topologies on $D([0,T])$, namely, $M_1$ and $J_1$ topologies. 
We will indicate the topology of weak convergence by simply writing `in $D([0,T],J_1)$' and `in $D([0,T],M_1)$'. The essential difference between
these two topologies is that while $M_1$ topology allows approximating processes make several jumps in short intervals of time to produce one bigger jump of the limiting process, while $J_1$ does not. See \cite{Whitt} for detailed
information on these topologies.

Below is our main result about the convergence of clock processes to an extremal process.

\begin{theorem}\label{thrm1} Assume Conditions A-D are satisfied. Set $\tn=\fn\gn$. Then for a.s.~random environment $\btau$, for any $T>0$
\begin{equation}\label{eqnthrm1}
\left(\frac{S_n(\cdot\;\rn)}{\tn}\right)^\an\overset{n\to\infty}\longrightarrow W(\cdot) \text{   in   } D([0,T];M_1)
\end{equation}
where $W$ is the extremal process generated by the distribution function $F(x)=e^{-1/x}$. Moreover, the above convergence holds in the stronger topology $D([0,T];J_1)$ if $f(n)=1$ and $\mathcal{K}_G=1$.
\end{theorem}

\begin{remark}
 Theorem \ref{thrm1} can be stated in a more general way. Namely, Condition A can be generalized such that (\ref{condA1}) is replaced by 
\begin{equation}
 b_nP_n(\tau_x\geq L^{\leftarrow}(u h_n))\overset{n\to\infty}\longrightarrow \frac{1}{u}
\end{equation}
and (\ref{condA2}) by
\begin{equation}
 P_n(\tau_x\geq L_n^{\leftarrow}(u d^\an h_n))\leq \frac{C\rhon}{u d^\an}
\end{equation}
where $h_n$ is a diverging scale. Here, for each $n\in\N$, $L_n$ is positive, as $u\to\infty$, $L_n(u)\to\infty$ and for any $\lambda>0$
\begin{equation}
 \frac{L_n(\lambda u)}{L_n(u)}\longrightarrow \lambda^{\an}
\end{equation}
in a mildly uniform way in $n$. Also, in this setup one can have $\an=0$ which would mean that $L_n$ is a slowly varying function for such an $n$. 
Then (\ref{eqnthrm1}) in Theorem 
\ref{thrm1} becomes
\begin{equation}
 \frac{L_n\left(S_n(\cdot \;\rn )/\fn\right)}{h_n}\Longrightarrow W(\cdot).
\end{equation}
Since such a setup makes the notation very difficult and does not bring a new conceptual insight we preferred to use the current setup. Finally,
note that, choosing $L_n(u)=u^{\an}$ and $\gn=h_n^{1/\an}$ gives Theorem \ref{thrm1}.

\end{remark}
\newcommand{\dn}[1]{{d_n({#1})}}
\newcommand{\un}[1]{{U_n({#1})}}
\newcommand{\sn}[1]{{s_n({#1})}}
\newcommand{\mn}[1]{{m_n({#1})}}

In the rest of the current section we prove Theorem \ref{thrm1}. 
We start by defining set of deep traps, very deep traps and shallow traps determined by the depth scale and the 
depth rate scale $\an$ and $\gn$: for $0<\ep<M$
\begin{equation}
\begin{array}{lll}
\tem:=\{x\in\V:\ep^{1/\an} \gn \leq \tau_x\leq M^{1/\an} \gn\;\}&\textbf{deep traps},\\
\;\tm:=\{x\in\V:\tau_x> M^{1/\an}\gn\}&\textbf{very deep traps},\\
\;\;\te:=\{x\in\V:\tau_x<\ep^{1/\an}\gn\}&\textbf{shallow traps.}
\end{array}
\end{equation}
Let $\dn{j}$ be sequence of times where a deep trap different from the last deep trap visited is found. We set $\dn{0}=0$ and for $j\in\N$ define recursively
\begin{equation}
\dn{j}:=\min\left\{i>\dn{j-1}:Y_n(i)\in \tem\setminus\{Y_n(\dn{j-1})\}\right\}.
\end{equation}
We define the process $(U_n(j):j\in\N_0)$
\begin{equation}
U_n(j):=Y_n(\dn{j}),
\end{equation}
and $\zeta_n$ as the last time the random walk finds a deep trap before $T\rn$:
\begin{equation}
\zn:=\max\{j:\;\dn{j}\leq T\rn\}.
\end{equation}
Let $\sn{j}$ be the time $X_n$ spends at $\un{j}$ between the first time it visits $\un{j}$ until it finds a different deep trap, that is,
\begin{equation}
\sn{j}=\sum_{i=\dn{j}}^{\dn{j+1}}e_i \tau_{\yi}\inc{\yi=\un{j}}.
\end{equation}
We call $\sn{j}$ the score of $\un{j}$. Note that, for given environment $\btaun$ and $U_n(j)$, the expectation of $\sn{j}$ over all the other random sources is 
\begin{equation}\label{meanscore}
\tau_{\un{j}}G_{\tem\setminus \{\un{j}\}}^n(\un{j},\un{j}).
\end{equation}
Finally, we define $m_n$, the record process of $s_n$, for we expect it to be a good approximation of the clock process. For technical convenience we set $m_n(0)=0$ and $m_n(1)=S_n(\dn{1})$, and define for $j\geq 2$
\begin{equation}\label{defmn}
\mn{j}:=\max_{i=1,\dots,j-1}s_n(i).
\end{equation}
As a first step of the proof, we need to control the distribution of depths of deep traps visited. We introduce the following notation:
\begin{equation}
\rho_a^b:=a^{-1}-b^{-1},\;\;\;\;\;\;0<a<b.
\end{equation}
We need the following lemma which is from \cite{BC08}, stated in a slightly different way.
\begin{lemma}\label{lebc}(Lemma 2.5 on page 304 in \cite{BC08})
Recall that $A_n$ is a percolation cloud with density $\rho\rhon$ and assume Condition C. Let $A_n^1$ and $A_n^2$ be such that $A_n^1\cup A_n^2 =A_n$ and $A_n^1\cap A_n^2=\emptyset$ and
\begin{equation}
\lim_{n\to\infty}\frac{|A_n^1|}{|A_n|}=\frac{\rho_1}{\rho},\;\;\;\lim_{n\to\infty}\frac{|A_n^2|}{|A_n|}=\frac{\rho_2}{\rho} \;\;\;\text{       with      }\;\;\;\rho_1+\rho_2=1.
\end{equation}
Then
\begin{equation}
\lim_{n\to\infty}\max_{x\in A_n\cup \{0\}}\left|\pp_x\Big(H_n(A_n^1\setminus \{x\})<H_n(A_n^2\setminus \{x\})\Big)-\frac{\rho_1}{\rho}\right|=0.
\end{equation}
\end{lemma}
Let $(\sigma_\epsilon^M(j):j\in\N)$ be an i.i.d. sequence of random variables with distribution
\begin{equation}\label{distsig}
P(\sigma_{\epsilon}^M(j)\geq u)=\frac{\rho_u^M}{\rho_\ep^M}.
\end{equation}
\begin{proposition}\label{asympindp}
The collection
$\left(\left(\tau_{\un{j}}/g(n)\right)^\an:j\in\N\right)$ converges weakly to $(\sigma_\epsilon^M(j):j\in\N)$ as $n\to\infty$.
\end{proposition}
\begin{proof}
Since $(\un{j}:j\in\N)$ is a Markov sequence, it is enough to show that for any $j\in\N$, as $n\to\infty$
\begin{equation}\label{arat}
\pp\left(\left(\tau_{\un{j}}/g(n)\right)^\an\geq u|\un{j-1}\right)\longrightarrow \frac{\rho_u^\ep}{\rho_\ep^M}.
\end{equation}
We have
\begin{equation}\label{yetti}
\pp\left(\left(\tau_{\un{j}}/g(n)\right)^\an\geq u|\un{j-1}\right)=\pp_x\Big(H_n(T_u^M(n)\setminus \{x\})<H_n(T_\ep^u(n)\setminus \{x\})\Big)
\end{equation}
where $x=\un{j-1}$. Note that $\tem=T_u^M(n)\cup T_\ep^u(n)$ and $T_u^M(n)\cap T_\ep^u(n)=0$. Moreover, by Condition A 
\begin{equation}
\lim_{n\to\infty}\frac{|T_u^M(n)|}{|\tem|}=\frac{\rho_u^M}{\rho_{\ep}^M}\;\;\;\text{ and }\;\;\;\lim_{n\to\infty}\frac{|T_\ep^u(n)|}{|\tem|}=\frac{\rho_\ep^u}{\rho_{\ep}^M}.
\end{equation}
Hence, (\ref{yetti}), Condition A and Lemma \ref{lebc} finish the proof. 
\end{proof}
\begin{lemma}\label{asindm}
The collection $\left(\left(\sn{j}/\tn\right)^\an:j\in\N\right)$ converges weakly to $(\sigma_\epsilon^M(j):j\in\N)$ as $n\to\infty$.
\end{lemma}
\begin{proof}
Using the Markov property of the random walk we see that
\begin{equation}\label{arasifir}
\begin{aligned}
&\;\;\pp\left((s_n(j)/\tn)^\an\geq u \big| s_n(1),\dots,s_n(j-1)\right)\\&=\int_\ep^M \pp\left(s_n(j)\geq u^{1/\an}\tn\big|\tau_{\un{j}}=v^{1/\an}\gn,s_n(1),\dots,s_n(j-1)\right)
\\&\;\;\;\;\;\;\;\;\;\;\;\;\;\;\times \pp\left((\tau_{\un{j}}/\gn)^\an\in dv \big| s_n(1),\dots,s_n(j-1)\right).
\end{aligned}
\end{equation}
Note that the first term inside the integral in the above display is equal to 
\begin{equation}
\pp\left(s_n(j)\geq u^{1/\an}\tn\big|\tau_{\un{j}}=v^{1/\an}\gn\right)\leq \left(\frac{u}{v}\right)^{1/\an} \frac{\gn}{\tn}\;G_{\tem\setminus\{\un{j}\}}^n(\un{j},\un{j})
\end{equation}
where we use Chebyshev's inequality and (\ref{meanscore}) for the upper bound. By Condition B we have $G_{\tem\setminus\{x\}}^n(x,x)=\mathcal{K}_{G}\fn(1+o(1))$ as $n\to\infty$ and the error term is uniformly bounded in $x\in\tem$. Hence, recalling that $\fn\gn=\tn$, if $u>v$ we have
\begin{equation}\label{arabir}
\pp\left(s_n(j)\geq u^{1/\an}\tn\big|\tau_{\un{j}}=v^{1/\an}\gn\right)\longrightarrow 0.
\end{equation}
Since $\sum_{i=\dn{j}}^{\dn{j+1}}\inc{\yi=\un{j}}$ is always at least one and $\tn=\fn\gn$, using Condition D part $(ii)$ one we have for $u<v$
\begin{equation}\label{araiki}
 \pp\left(s_n(j)\geq u^{1/\an}\tn\big|\tau_{\un{j}}=v^{1/\an}\gn\right)\geq \pp(e_1v^{1/\an}\geq u^{1/\an}\fn)\longrightarrow 1.
\end{equation}
Using (\ref{arabir}) and (\ref{araiki}) in (\ref{arasifir}), than applying Proposition \ref{asympindp} finishes the proof.
\end{proof}

We next show that the contribution from the shallow traps are negligible.
\begin{proposition}\label{shallowt}
For any $a,d>0$ and $T>0$, for any $\delta>0$ given, $\btau$ a.s. for $n$ large enough
\begin{equation}
\E\left[\sum_{i=0}^{T\rn} \tau_{\yn{i}}e_i \inc{\yn{i}\in T^a(n)}\Big|\;\btau\right]\leq (a+d)^{1/\an}\tn.
\end{equation}
\end{proposition}
\begin{proof}
We first prove that $\btau$ a.s. for $n$ large enough, for all $j\in\N$
\begin{equation}\label{aug}
\E\left[\sum_{i=1}^{T\rn}e_i\tau_{Y_n(i)}\inc{Y_n(i)\in T_{a2^{-j\an}}^{a 2^{-j\an+\an}}(n)}\Big|\;\btau\right]\leq (a+d/2)^{1/\an}\tn 2^{j\an-j+1}.
\end{equation}
By Condition A we have for all $n\in\N$ and $j\in\N$
\begin{equation}\label{dumonde}
p_{n,j}:=P\Big(\inc{x\in T_{a2^{-j\an}}^{a 2^{-j\an+\an}}(n)}\Big)\leq \frac{C 2^{j\an}}{a}\rhon.
\end{equation}
We have
\begin{align}
\nonumber&\;\;P\left(\E\left[\sum_{i=1}^{T\rn}e_i\tau_{Y_n(i)}\inc{Y_n(i)\in T_{a2^{-j\an}}^{a 2^{-j\an+\an}}(n)}\Big|\;\btau\right]\geq (a+d/2)^{1/\an}\tn 2^{j\an-j+1}\right)\\
\nonumber&=P\left(\sum_{x\in\V_n}G_{T\rn}(0,x)\tau_x \inc{x\in T_{a2^{-j\an}}^{a 2^{-j\an+\an}}(n)}\geq (a+d/2)^{1/\an}\tn 2^{j\an-j+1}\right)\\
\label{cello}&\leq P\left(\sum_{x\in\V_n}G_{T\rn}(0,x)\inc{x\in T_{a2^{-j\an}}^{a 2^{-j\an+\an}}(n)}\geq \left(\frac{a+d/2}{a}\right)^{1/\an}2^{j\an}\fn\right)
\end{align}
Using exponential Chebyshev inequality with $\lambda_n$ from Condition D, (\ref{dumonde}) and the bound $\log(1+x)\leq x,$ for $x\geq 0$; we conclude that (\ref{cello}) is bounded above by
\begin{equation}\label{tout}
\exp\left(-\lambda_n \left(\frac{a+d/2}{a}\right)^{1/\an} 2^{j\an}\fn+\frac{C 2^{j\an}}{a}\rhon \sum_{x\in\V_n}\big(e^{\lambda_n G_{T\rn}(0,x)}-1\big)\right).
\end{equation}
Now using Condition D and that $\rhon\rn=\fn$, we arrive at that (\ref{tout}) is bounded above by
\begin{equation}
\exp\left(-2^{j\an}\lambda_n \fn D_n\right)
\end{equation}
for all $n$ large enough, for all $j\in\N$, where
\begin{equation}
D_n:=\left\{\left(\frac{a+d/2}{a}\right)^{1/\an} -\frac{C\mathcal{K}_s T}{a}\right\}.
\end{equation}
We have for all $n$
\begin{equation}\label{matin}
\sum_{j=1}^\infty \exp\left(-2^{j\an}\lambda_n \fn D_n\right)\leq \exp\left(-2^{\an}\lambda_n \fn D_n\right)+\frac{\exp\left(-2^{\an}\lambda_n \fn D_n\right)}{\an(\log 2)2^{\an}\lambda_n \fn D_n}
\end{equation}
Since both $\an D_n\to\infty$ and $D_n\to\infty$ as $n\to\infty$; by the part of Condition D that states that $\sum_{n=1}^\infty \exp(-c \lambda_n f(n))<\infty$ for some $c>0$, we can conclude that the right hand side of (\ref{matin}) is summable in $n$. Hence, Borel-Cantelli lemma yields (\ref{aug}). Finally, summing over $j$ finishes the proof of Lemma \ref{shallowt}. 
\end{proof}
The next proposition shows that the sum of scores is dominated by the largest score.
\begin{proposition}\label{comp}
For any $\delta>0$ given, for any $\ep$ small and $M$ large enough, $\btau$ a.s. there exists a constant $K\geq 1$ such that for $n$ large enough
\begin{equation}
\pp\big(S_n(\dn{j})\leq K m_n(j),\;\;\forall j \text{  s.t.  }\dn{j}\leq T\rn\big)\geq 1-\delta.
\end{equation}
\end{proposition}
\begin{proof}

Recall that $\zn$ is the last time the random walk finds a deep trap before $T\rn$. Note that for $j=0,1$ we have by definition $S_n(\dn{j})=\mn{j}$. Hence, we assume that $\zn\geq 2$ and prove that
\begin{equation}
\pp\big(S_n(\dn{j})\leq K m_n(j),\;\;\forall j=2,\dots,\zn|\btau\big)\geq 1-\delta.
\end{equation}
We define the sequence of events that the random walk cannot find a very deep trap before time $T\rn$:
\begin{equation}
I_n^1:=\{H_n(T_M(n))\geq T \rn \}.
\end{equation}
We can use Condition C with $\rho=1/M$ to conclude that for $M$ large enough
\begin{equation}
\pp(I_n^1\;|\btau)\geq 1-\delta/4.
\end{equation}
For $d>0$ define the sequence of events
\begin{equation}
I_n^2:=\{s_n(1)\geq (\ep+d)^{1/\an}\tn\}.
\end{equation}
Using Proposition \ref{asympindp}, for $d$ small enough, $\btau$ a.s. for $n$ large enough
\begin{equation}\label{firstscore}
\pp(I_n^2\;|\btau)\geq 1-\delta/4.
\end{equation}
We define another sequence of events 
\begin{equation}\label{smalls}
I_n^3:=\left\{\sum_{i=0}^{T\rn}e_i \tau_{\yn{i}}\inc{\yn{i}\in \te}\leq (\ep+d/2)^{1/\an}\tn\right\}.
\end{equation}
By Proposition \ref{shallowt} we have $\btau$ a.s. for $n$ large enough
\begin{equation}
\pp(I_n^3\;|\btau)\geq 1-\delta/4.
\end{equation}
Finally, defining
\begin{equation}
I_n^4:=\{\zn\leq K-1\},
\end{equation}
using Condition C, for $K$ large enough, $\btau$ a.s. for $n$ large enough
\begin{equation}
\pp(I_n^4\;|\btau)\geq 1-\delta/4.
\end{equation}
For $d$ and $M$ chosen as above, for $j=2,\dots,\zn$, we partition the sum $S_n(\dn{j})$ as follows
\begin{equation}\label{sonsum}
S_n(\dn{j})=\sum_{i=0}^{\dn{j}-1}e_i\tau_{\yn{i}}\inc{\yn{i}\in\te}+ \sum_{i=0}^{\dn{j}-1}e_i\tau_{\yn{i}}\inc{\yn{i}\in\tm}+\sum_{k=1}^j s_n(k).
\end{equation}
Recall that $\mn{j}=\max_{i=1,\dots,j-1}s_n(i)
$. Let $I_n:=I_n^1\cap I_n^2\cap I_n^3\cap I_n^4$. We have $\btau$ a.s. for $n$ large enough $\pp(I_n)\geq 1-\delta$. On the event $I_n$, the first term on the right hand side of is bounded above by $m_n(j)$ for any $j\geq 1$ using (\ref{firstscore}) and (\ref{smalls}); the second term is 0 since on $I_n^1$ no deep trap has been found; the third term is bounded by $K-1$ since on $I_n^4$ $\zn\leq K-1$. Hence, on $I_n$, $S_n(\dn{j})$ is bounded $Km_n(j)$ and we are finished with the proof.
\end{proof}
\begin{proposition}\label{keyp}
For any $\delta>0$ given, for $\ep$ small and $M$ large enough, $\btau$ a.s. for large enough $n$
\begin{equation}
\pp\left(\max_{j\leq \zn} \left|\left(\frac{S_n(\dn{j})}{\tn}\right)^\an-\left(\frac{\mn{j}}{\tn}\right)^\an\right|\geq \delta\Big|\btau\right)\leq \delta.
\end{equation}
\end{proposition}
\begin{proof}
By definition we have $\mn{j}\leq S_n(\dn{j})$ for any $j$. Using Proposition \ref{comp} we can find a positive constant $K$ s.t. $\btau$ a.s. for $n$ large enough
\begin{equation}\label{dnn1}
\pp\big(S_n(\dn{j})\leq K\mn{j},\;\forall j=1,\dots,\zn|\btau\big)\geq 1-\delta/4.
\end{equation}
For $K$ as above, using Lemma \ref{asindm} we can choose a small enough $d$ so that $\btau$ a.s. for $n$ large enough
\begin{equation}\label{dnn2}
\pp\big(\mn{j}\in [(\ep-d)^{1/\an}\tn,(M+d)^{1/\an}\tn],\forall j=1,\dots,K\big)\geq 1-\delta/4.
\end{equation}
Finally, for $K$ and $d$ as above, for $n$ large enough
\begin{equation}\label{dnn3}
|K^\an-1|(M+2d)\leq \delta/2.
\end{equation}
We denote by $I_n$ the intersection of the events inside the probability displays in (\ref{dnn1}) and (\ref{dnn2}). Then using (\ref{dnn1}), (\ref{dnn2}) and (\ref{dnn3}) we get
\begin{equation}
\begin{aligned}
&\;\;\;\;\pp\left(\max_{j\leq \zn} \left|\left(\frac{S_n(\dn{j})}{\tn}\right)^\an-\left(\frac{\mn{j}}{\tn}\right)^\an\right|\geq \delta\Big|\btau\right)
\\&\leq \delta/2+\pp\left(\max_{j\leq \zn} \left|\left(\frac{K\mn{j}}{\tn}\right)^\an-\left(\frac{\mn{j}}{\tn}\right)^\an\right|\geq \delta\Big|I_n, \btau\right)= \delta/2.
\end{aligned}
\end{equation}
\end{proof}
\begin{proof}[Proof of Theorem \ref{thrm1}] 
\newcommand{\zni}{{\zeta_n(t_i)}}
We start the proof with the proof of convergence of the finite dimensional distributions. Let $0=t_0\leq t_1\leq t_2\leq\cdots\leq t_k\leq T$ and $0\leq x_1\leq x_2\leq \cdots\leq x_k$ be given. Consider the random variables $\zni,i=1,\dots,k$. For convenience, we denote by $N(\lambda)$, a Poisson random variable with mean $\lambda \rho_\ep^M$. Let $\{N(t_i-t_{i-1}):i=1,\dots,k\}$ be an independent collection. By Condition C, we have as $\btau$ a.s. $n\to\infty$ 
\begin{equation}\label{indpois}
\left\{\zeta_n(t_i)-\zeta_n(t_{i-1}):i=1,\dots,k\right\}\Longrightarrow \left\{N(t_i-t_{i-1}):i=1,\dots,k\right\}
\end{equation}
where $\Longrightarrow$ stands for convergence in distribution. For ease of notation we defined the rescaled clock process
\begin{equation}
\bar{S}_n(t):=\left(\frac{S_n(t \rn)}{\tn}\right)^\an.
\end{equation}

\newcommand{\bsn}[1]{\bar{S}_n(#1)}
\newcommand{\bsnn}{\bar{S}_n}

By Proposition \ref{keyp} and using Proposition \ref{shallowt} as in the proof of Proposition \ref{keyp}, we can conclude that $\btau$ a.s. the sequence of events that for all $i=0,1,\dots,k$
\begin{equation}
\left(\frac{\mn{\zni}}{\tn}\right)^\an-\delta \leq \bsn{t_i} \leq \left(\frac{\mn{\zni}}{\tn}\right)^\an+\delta
\end{equation}
has probability larger than $1-\delta$ for all $n$ large enough. Hence, the sequence of finite dimensional distributions
\begin{equation}
\pp\left(\bsn{t_1}\leq x_1,\dots,\bsn{t_k}\leq x_k\Big|\btau\right)
\end{equation}
is bounded above by 
\begin{equation}\label{extup}
\pp\left(\left(\frac{m_n(\zeta_n(t_1))}{\tn}\right)^\an\leq x_1+\delta,\dots,\left(\frac{m_n(\zeta_n(t_k))}{\tn}\right)^\an\leq x_k+\delta\Big|\btau\right)
\end{equation}
and below by
\begin{equation}
\pp\left(\left(\frac{m_n(\zeta_n(t_1))}{\tn}\right)^\an\leq x_1-\delta,\dots,\left(\frac{m_n(\zeta_n(t_k))}{\tn}\right)^\an\leq x_k-\delta\Big|\btau\right).
\end{equation}
We prove only the upper bound for the lower bound can be achieved similarly. By Lemma \ref{asindm} and (\ref{indpois}) (also recall (\ref{distsig})), $\btau$ a.s. as $n\to\infty$, the sequence of probability terms in (\ref{extup}) converges to
\begin{equation}\label{istan}
\E\left[\left(\rho_{\ep}^{x_1+\delta}/\rho_{\ep}^M\right)^{N(t_1)}\right]\E\left[\left(\rho_{\ep}^{x_2+\delta}/\rho_{\ep}^M\right)^{N(t_2-t_1)}\right]\cdots \E\left[\left(\rho_{\ep}^{x_k+\delta}/\rho_{\ep}^M\right)^{N(t_k-t_{k-1})}\right].
\end{equation}
A simple calculation yields that for any $x,\lambda\geq 0$ 
\begin{equation}
\E\left[\left(\rho_\ep^x/\rho_\ep^M\right)^{N(\lambda)}\right]=\exp(-\lambda \rho_{x}^M).
\end{equation}
Hence, (\ref{istan}) is equal to
\begin{equation}
\exp\left(-\frac{t_1}{x_1+\delta}\right)\exp\left(-\frac{t_2-t_1}{x_2+\delta}\right)\cdots \exp\left(-\frac{t_k-t_{k-1}}{x_k+\delta}\right).
\end{equation}
Finally, letting $\delta\to 0$ finishes the proof of the convergence of finite dimensional distributions.

\newcommand{\vepsilon}{\varepsilon}

For tightness characterizations we need the following definitions:
\begin{align*}
&w_f(\delta)=\sup\left\{\min\big(|f(t)-f(t_1)|,|f(t_2)-f(t)|\big):t_1\leq t\leq t_2\leq T,t_2-t_1\leq \delta\right\},\\
&w'_f(\delta)=\sup\{\inf_{\alpha \in [0,1]}|f(t)-(\alpha f(t_1)+(1-\alpha)f(t_2))|:t_1\leq t \leq t_2\leq T, t_2-t_1\leq \delta\},\\
& v_f(t,\delta)=\sup \{|f(t_1)-f(t_2)|:t_1,t_2\in [0,T]\cap (t-\delta,t+\delta)\}.
\end{align*}
Following is from Theorem 12.12.3 of \cite{Whitt} and Theorem 15.3 of \cite{bill}.
\begin{theorem}\label{kiki}
 The sequence of probability measures $\{P_n\}$ on $D([0,T])$ is tight in the $M_1$-topology if

(i) For each positive $\vepsilon$ there exists a $c$ such that
\begin{equation}
 P_n[h:||h||_\infty \geq c]\leq \vepsilon, \hspace{0.4in} n\geq 1
\end{equation}

(ii) For each $\vepsilon>0$ and $\eta >0$, there exists a $\delta$, $0<\delta<T$, and an integer $n_0$ such that
\begin{equation}\label{budane}
 P_n[h: w'_h(\delta)\geq \eta]\leq \vepsilon, \hspace{0.4in} n\geq n_0
\end{equation}
and
\begin{equation}\label{ends}
 P_n[h:v_h(0,\delta)\geq \eta]\leq \vepsilon \; \text{and} \; P_n[h:v_h(T,\delta)\geq \eta]\leq \vepsilon, \;\; n\geq n_0
\end{equation}
Moreover, the same is true for $J_1$ topology with $w'_h(\delta)$ in (\ref{budane}) replaced by $w_h(\delta)$.
\end{theorem}

For the first claim in Theorem \ref{thrm1}, we check tightness in $M_1$ topology, using Theorem \ref{kiki}. Since $\bsn{t}$ is non-decreasing in $t$, to check condition (i) it is enough to check that $\bsn{T}$ is tight. In this case, the convergence of fixed time distribution gives the desired result.

Since for monotone functions $w'_{h}(\delta)$ is 0 in part (ii) of Theorem \ref{kiki} we only need to control $v_{\bsnn}(0,\delta)$ and $v_{\bsnn}(T,\delta)$. Again, using the monotonicity, controlling $v_{\bsnn}(0,\delta)$ boils down to check that $\mathbb{P}[\bsnn(\delta)\geq \eta]\leq \vepsilon$ for small enough $\delta$ and large enough $n$. By convergence of the fixed time distribution it is enough to check $\mathbb{P}(W(\delta)\geq \eta/2)\leq \vepsilon/2$. Since $\pp(W(\delta)\geq \eta/2)=1-e^{-2\delta/\eta}$, this claim is true for small enough $\delta$. Similarly, to control $v_{\bsnn}(T,\delta)$ it is enough to find $\eta$ small enough so that
\begin{equation}\label{fosfos}
 \mathbb{P}[W(T)-W( T-\delta)\geq \eta]\leq \vepsilon/2.
\end{equation}
Observe that by (\ref{poschr})
\begin{equation}
 \mathbb{P}[W(T)-W(T-\delta)=0]=\frac{T-\delta}{T},
\end{equation}
then
\begin{equation*}
 \mathbb{P}[W(T)-W( T-\delta)\geq \eta]\leq 1-\mathbb{P}[W(T)-W(T-\delta)=0]=\frac{\delta}{T}.
\end{equation*}
Hence, (\ref{fosfos}) follows by taking $\delta\leq T\vepsilon/4$. Hence, we are finished with the proof of the first part of Theorem \ref{thrm1}.

Now we assume that $\fn=1$ and $\mathcal{K}_G=1$ and prove the tightness of $\bsnn$ in $J_1$ topology. We only need to check that (\ref{budane}) with $w_{\bsnn}$. It is enough to show that $\btau$ a.s. as $n\to\infty$ 
\begin{equation}\label{esss}
\max_{x\in\tem}\pp_x\left(Y_n(i)\notin \tem,\; i=1,\dots \delta \rn\right)\leq \vepsilon
\end{equation}
for $\delta$ small enough. Since Condition B is satisfied with $\fn=\mathcal{K}_G=1$ we have
\begin{equation}
\max_{x\in\tem}\pp_x(H'_n(x)<H_n(\tem\setminus\{x\}))\longrightarrow 0
\end{equation}
where
\begin{equation}
H_n'(x):=\min\{i\geq 1: Y_n(i)=x\}
\end{equation}
Combining this with Condition C yields (\ref{esss}). Hence, we have proved the $J_1$ convergence.

\end{proof}

\begin{section}{Extremal Aging}\label{sc3}
\newcommand{\qn}[1]{{q_n({#1})}}
\newcommand{\kn}[1]{{k_n({#1})}}
\newcommand{\Vn}[1]{{V_n({#1})}}

Next, we give two extra conditions that ensure extremal aging occurs. For this purpose we need to control the random walk between record sites. First, we define the sequence of times when a new record site is found. We define $q_n(1)=1$ and for $j\geq 2$
\begin{equation}
\qn{j}=\min\{i>\qn{j-1}:\;\mn{i}>\mn{\qn{j-1}}\}.
\end{equation}
Hence, $\qn{j}$th score is greater than all the scores before it. To keep track of the times when the random walk visits a record site we introduce
\begin{equation}
k_n(j)=\dn{\qn{j}},\;\;\; j\in\N.
\end{equation}
Finally, we define the process $V_n(j)$ that records the trajectory of $Y_n$ restricted to deep traps whose score is a record:
\begin{equation}
\Vn{j}=U_n(\qn{j}).
\end{equation}
We fix numbers $a$ and $b$ with $0<a<b$. We set $\xi_n=T\rn$, the number of steps of the random walk that we observe. Later we will choose $T$ large enough so that the clock process reaches the level $a^{1/\an}\tn$. The first condition is that in the time between the record site $\Vn{j}$ is found and the next record site $\Vn{j+1}$ is found, with a high probability the trap model is at $\Vn{j}$. More precisely, let $\tn$ be a deterministic sequence of times satisfying $(a/2)^{1/\an}\tn\leq t'_n\leq b^{1/\an}\tn$ and let $\delta>0$. We define $j_n\in\N$ by
\begin{equation}\label{tnsqn}
\left(\frac{S_n(\kn{j_n})}{\tn}\right) +\delta \leq \left(\frac{t'_n}{\tn}\right)^\an \leq \left(\frac{S_n(\kn{j_n+1})}{\tn}\right) -\delta
\end{equation}
and $j_n=\infty$ if the above inequalities are not satisfied by an integer. Let $A_n(\delta)$ be the event
\begin{equation}\label{ansqn}
A_n(\delta):=\{0<j_n<\zn\}.
\end{equation}

\noindent{\bf Condition 1:} For any $\delta>0$ it is possible to choose $\epsilon$ small and $M$ large enough so that $\btau$ a.s. for $n$ large enough
\begin{equation}
\pp(X_n(t'_n)=V_n(j_n)|A_n(\delta),\btau)\geq 1-\delta.
\end{equation}
The second condition states that there are no repetitions among record sites.

\noindent{\bf Condition 2:} For any $\ep$ and $M$, $\btau$ a.s.
\begin{equation}
\lim_{n\to\infty}\pp(\exists i,j \text{ s.t. } i\not=j,\qn{i},\qn{j}\leq \zn, \Vn{i}=\Vn{j}|\btau)=0.
\end{equation}
As discussed before, our choice of he two-time correlation function $R_n$ is
\begin{equation}
R_n(t_1,t_2):=\pp(X_n(t_1)=X_n(t_2)|\btaun)
\end{equation}
Now we are ready to state our extremal aging result.
\begin{theorem}\label{thrm2}
Assume that Conditions A-D and 1-2 hold and let $0<a<b$. Then $\btau$ a.s. 
\begin{equation}
\lim_{n\to\infty}R_n(a^{1/\an}\tn,b^{1/\an}\tn|\btau)=\frac{a}{b}.
\end{equation}
\end{theorem}

\begin{proof}[Proof of Theorem \ref{thrm2}]
We first calculate the probability that the record process jumps over an interval. We define the range of $m_n$ by
\begin{equation}
\rr(m_n):=\left\{\left(\frac{\mn{j}}{\tn}\right)^\an:\;j\in\N\right\}.
\end{equation}
Note that, $m_n$ depends on $\epsilon$ and $M$ for the choice of $\tem$.
\begin{lemma}\label{extjump} For $\ep<a<b<M$, $\btau$ a.s.
\begin{equation}
\lim_{n\to\infty}\pp\Big(\rr(m_n)\cap [a,b]=\emptyset\Big)=\frac{\rho_b^M}{\rho_a^M}.
\end{equation}
\end{lemma}
\begin{proof}
By definition we have $m_n(1)=S_n(\dn{1})$. Since $\ep<a$, by Lemma \ref{shallowt}, as $n\to\infty$ the probability that $\mn{1} \geq a^{1/\an} \tn$ vanishes. Hence, using the weak convergence result in Lemma \ref{asindm} and the fact that the distribution of $\sigma_\ep^M$ has no atoms
\begin{equation}\label{extj}
\lim_{n\to\infty} \pp\Big(\rr(m_n)\cap [a,b]=\emptyset\Big)=\pp\Big(\big\{\max_{i=1,\dots,j}\sigma_{\ep}^M(i):\;j\in\N\big\}\cap [a,b]=\emptyset\Big)
\end{equation}
Since $(\sigma_\ep^M(i):i\in\N)$ is an i.i.d. sequence we have (\ref{extj}) equal to
\begin{equation}
\begin{aligned}
&\pp\Big(\sigma_\ep^M(1)\geq b\Big)+\sum_{j=0}^\infty \pp\Big(\max_{i=1,\dots,j}\sigma_{\ep}^M(i)\leq a\Big)\pp\Big(\sigma_\ep^M(j+1)\geq b\Big)
\\&= \sum_{j=0}^\infty \pp(\sigma_\ep^M(1)\leq a)^j \pp(\sigma_\ep^M(1)\geq b)=\frac{\pp(\sigma_\ep^M(1)\geq b)}{\pp(\sigma_\ep^M(1)\geq a)}=\frac{\rho_b^M}{\rho_{a}^M}.
\end{aligned}
\end{equation}
\end{proof}

\newcommand{\dt}{\text{dist}}

Let $\delta>0$ be small enough so that $\ep<a-2\delta<b+\delta<M$. Define the sequence of events $I_n^1,I_n^2$ and $I_n^3$ as follows:
\begin{equation}
\begin{aligned}
&I_n^1:=\Big\{\dt(a,\rr(m_n))\leq \delta\;\;\text{ or } \dt(b,\rr(m_n))\leq\delta\Big\},
\\& I_n^2:=\Big\{\dt(a,\rr{m_n})\geq \delta,\;\dt(b,\rr(m_n))\geq \delta \text{ and } [a+\delta,b-\delta]\cap \rr(m_n)\not=\emptyset\Big\},
\\&I_n^3:=\Big\{[a-\delta,b+\delta]\cap \rr(m_n)=\emptyset\Big\}
\end{aligned}
\end{equation}
We also define $F_n:=\big\{m_n(\zn)\geq b^{1/\an}\tn\big\}$. Since a.s. $W(t)\to\infty$ as $t\to\infty$, using Theorem \ref{thrm1} we can choose $T>0$ large enough so that $\btau$ a.s. for $n$ large enough
\begin{equation}\label{ufak}
\pp(F_n|\btau)\geq 1-\delta/4.
\end{equation}
Finally, let us define the sequence of events we want to approximate 
\begin{equation}
G_n:=\Big\{X_n(a^{1/\an}\tn)=X_n(b^{1/\an}\tn)|\btau\Big\}.
\end{equation}
We want to show that $G_n$ can be well approximated by $I_n^3$. Since the jumps of $m_n$ have continuous distribution we have $(I_n^3)^c=I_n^1\cup I_n^2$ and $(I_n^2)^c\subset I_n^1\cup I_n^3$. This yields to
\begin{equation}
\pp(G_n\cap I_n^3)\leq \pp(G_n)\leq \pp(I_n^1)+\pp(I_n^3)+\pp(G_n \cap I_n^2).
\end{equation}
On $F_n$ and $I_n^2$ there exist $j_1$ and $j_2$ with $0<j_1<j_2$ and $q_n(j_1),q_n(j_2)<\zn$ such that
\begin{equation}
\begin{aligned}
&\left(\frac{m_n(q_n(j_1))}{\tn}\right)^\an+\delta \leq a \leq \left(\frac{m_n(q_n(j_1+1))}{\tn}\right)^\an -\delta \\
&\left(\frac{m_n(q_n(j_2))}{\tn}\right)^\an+\delta \leq a \leq \left(\frac{m_n(q_n(j_2+1))}{\tn}\right)^\an -\delta.
\end{aligned}
\end{equation}
Using Proposition \ref{keyp} this yields to
\begin{equation}
\begin{aligned}
&\left(\frac{S_n(q_n(j_1))}{\tn}\right)^\an+\delta/2 \leq a \leq \left(\frac{S_n(q_n(j_1+1))}{\tn}\right)^\an -\delta/2 \\
&\left(\frac{S_n(q_n(j_2))}{\tn}\right)^\an+\delta/2 \leq a \leq \left(\frac{S_n(q_n(j_2+1))}{\tn}\right)^\an -\delta/2.
\end{aligned}
\end{equation}
Hence, by Condition 1 we have $\btau$ a.s. for $n$ large enough 
\begin{equation}
\pp(X_n(a^{1/\an})=V_n(j_1),X_n(b^{1/\an})=V_n(j_2)|I_n^2,F_n,\btau)\geq 1-\delta.
\end{equation}
Combining this with Condition 2 and (\ref{ufak}) we get
\begin{equation}
\pp(G_n\cap I_n^2)\leq \delta.
\end{equation}
Similarly on $F_n\cap I_n^3$ there exists a $j_1$ such that
\begin{equation}
\left(\frac{m_n(q_n(j_1))}{\tn}\right)^\an+\delta \leq a \leq b \leq \left(\frac{m_n(q_n(j_1+1))}{\tn}\right)^\an -\delta.
\end{equation}
Hence, by Condition 1 we have $\btau$ a.s. for $n$ large enough
\begin{equation}
\pp(X_n(a^{1/\an}\tn)=X_n(b^{1/\an}\tn)=V_n(j_1)|\btau)\geq 1-\delta,
\end{equation}
and since $\pp(G_n)\geq \pp(X_n(a^{1/\an}\tn)=X_n(b^{1/\an}\tn)=V_n(j_1)|\btau)$ we get
\begin{equation}
\pp(I_n^3)-\delta\leq \pp(I_n^3\cap G_n).
\end{equation}
By Lemma \ref{extjump} we have $\pp(I_n^1)\to 0$ and $\pp(I_n^3)\to \rho_{b}^M/\rho_a^M$ as $\delta\to 0$. Finally, taking $M\to\infty$ and $\delta\to 0$ finishes the proof.
\end{proof}
\end{section}

\begin{section}{Extremal Aging for Random Energy Model}\label{sc4}
The Random Energy Model was first introduced by Bernard Derrida in \cite{Der1} as an exactly solvable mean field spin glass model. The state space is $\V_n=\{-1,+1\}^n$, the $n$ dimensional hypercube. To each configuration $\sigma\in\V_n$ is attached a random Hamiltonian $H_n(\sigma)$. The choice for the energy landscape in REM is that of i.i.d. Gaussians with mean zero and variance $n$. More precisely, $H_n(\sigma)=-\sqrt{n}Z_\sigma$ where
\begin{equation}\label{seqgauss}
\Big\{Z_\sigma:\sigma\in\V_n\Big\}
\end{equation}
is an i.i.d. sequence of standard Gaussian random variables. Let $\beta>0$ be the inverse temperature (later we will let $\beta$ vary with the dimension $n$), then the Gibbs measure at inverse temperature $\beta$ is given by
\begin{equation}
\mu_n(\sigma)=\frac{1}{\mathcal{Z}_{n,\beta}}e^{-\beta H_n(\sigma)}=\frac{1}{\mathcal{Z}_{n,\beta}} e^{\beta\sqrt{n} Z_\sigma},\;\;\;\sigma\in\V_n
\end{equation}
where $\mathcal{Z}_{n,\beta}$ is the usual partition function.

\newcommand{\calg}{\mathcal{G}}
\newcommand{\cale}{\mathcal{E}_n}

We will study the trap model dynamics of REM. In our setting, the graph is $\mathcal{G}_n=(\V_n,\cale)$ where the set of edges $\cale$ is given by
\begin{equation}
\mathcal{E}_n=\Big\{(\sigma,\sigma'):\frac{1}{2}\sum_{i=1}^n|\sigma_i-\sigma_i'|=1\Big\}.
\end{equation}
In other words, two configurations in $\V_n$ are neighbors if they differ only at one spin. For the trapping landscape $\{\tau_\sigma:\sigma\in\V_n\}$ we choose the Gibbs waits
\begin{equation}
\tau_\sigma:=\exp(-\beta H_n(\sigma))=\exp(\beta\sqrt{n} Z_\sigma),\;\;\;\sigma\in \V_n.
\end{equation}
As before we denote the corresponding trap model by $(X_n(t):\;t\geq 0)$. It is trivial that the Gibbs measure of REM is the unique invariant measure of $X_n$. We want to note there that in the literature this type of dynamics is sometimes called Random Hopping Time dynamics.

We will study the dynamics of REM model in the following context. We let the temperature vary by the volume of the system, hence, here after we parameterize the inverse temperature by $n$ and replace $\beta$ with $\beta_n$ in the above equations. Let us rewrite the trapping landscape in this new setup
\begin{equation}\label{trgas}
\btau:=\Big\{\exp(\beta_n \sqrt{n} Z_\sigma):\;\sigma\in\V_n\Big\}
\end{equation}
where the sequence $\{Z_\sigma:\sigma\in\V_n\}$ is as in (\ref{seqgauss}). As before, we denote by $Y_n$ the simple random walk on the hypercube $\V_n$ and the clock process by $S_n$.

\newcommand{\cn}{c_n}
\newcommand{\ann}{\alpha_n^2}
Next we choose the time scales that we will observe the dynamics at. We set first our depth rate scale $\an$; as in Section \ref{scClock}, $\an\to 0$ as $n\to\infty$. We consider the following scales:
\begin{equation}\label{scales}
\gn=\tn=\exp(\an\beta_n^2 n),\;\;\rhon^{-1}=\rn=\an\beta_n\sqrt{2\pi n}\exp(\alpha_n^2\beta_n^2 n/2),\;\;\fn=1.
\end{equation}
We will always assume that
\begin{equation}\label{assum0}
\limsup_{n\to\infty}\an\bn<\sqrt{2\log 2}.
\end{equation}
Furthermore, we will assume that $\an$ and $\beta_n$ are chosen such that
\begin{equation}\label{assum}
 n\log n \ll r_n\;\;\text{ as }\;n\to\infty.
\end{equation}
The following theorem describes the extremal aging for the dynamics of REM.
\begin{theorem}\label{thrm3}
For a.s. random environment $\btau$
\begin{itemize}
\item[(i)] for any $T>0$, as $n\to\infty$
\begin{equation}
\hspace{0.99in}\left(\frac{S_n\big(\cdot\;\rn\big)}{\tn}\right)^\an \Longrightarrow W(\cdot)\;\;\;\;\;\;\;\; \text{    in  }\;\; D([0,T],J_1)
\end{equation}
where $W$ is the extremal process generated by the distribution function $F(x)=e^{-1/x}$,

\item[(ii)] for any $0<a<b$ as $n\to\infty$
\begin{equation}
R_n\Big(a^{1/\an}\tn,b^{1/\an}\tn\big|\btau\Big)\longrightarrow\frac{a}{b}.
\end{equation}
\end{itemize}
\end{theorem}

\vspace{0.1in}

\begin{remark}
Note that the ratio of the two times $\tn a^{1/\an}$ and $\tn b^{1/\an}$ diverges with $n$ but  the logarithmic ratio $\log (\tn a^{1/\an})/\log (\tn b^{1/\an})$ converges to 1 as $n\to\infty$, due to (\ref{assum}). 
Hence, we can think of the extremal aging result of Theorem \ref{thrm3} as ``just before aging". 
\end{remark}
\begin{remark}
Let us describe the cases that Theorem \ref{thrm3} covers. For constant temperature case $\beta_n\equiv \beta>0$, the time scales in Theorem \ref{thrm3} consist of subexponential time scales (in $n$). For the case that $1\ll\beta_n$, that is, zero temperature dynamics, depending on $\beta_n$ and $\an$, it covers superexponential, exponential and subexponential time scales. Finally, for $\beta_n\ll 1$, time scales covered are subexponential. 
\end{remark}
\begin{proof}[Proof of Theorem \ref{thrm3}]
In order to prove Theorem \ref{thrm3} we will check Conditions A-D and Conditions 1-2; and to achieve convergence in $J_1$ we will check Condition B with $\mathcal{K}_G=1$.

\vspace{0.05in}

\noindent{\bf Condition A:} It is well-known for a standard Gaussian random variable $Z$ that
\begin{equation}\label{e1}
\;\;\;\;P(Z\geq u)=\frac{1}{u\sqrt{2\pi}}e^{-u^2/2}(1+o(1))\;\;\;\text{ as  }u\to\infty,
\end{equation}
and
\begin{equation}\label{e2}
P(Z\geq u)\leq \frac{1}{u\sqrt{2\pi}}e^{-u^2/2},\;\;\;\forall u>0.
\end{equation}
By (\ref{assum}) we have $1\ll \an\beta_n\sqrt{n}$. Hence, using (\ref{trgas}) and (\ref{e1})  
\begin{equation}
\begin{aligned}
\Big(\an&\bn\pin\Big) e^{\ann\bnn n/2}P(\tau_x\geq u^{1/\an}e^{\an\bnn n})\\&=\Big(\an\bn\pin\Big) e^{\ann\bnn n/2} P\Big(Z\geq \an\bn\sqn +\frac{\log u}{\an\bn\sqn}\Big)\underset{n\to\infty}\longrightarrow \frac{1}{u}.
\end{aligned}
\end{equation}
This proves the first part of Condition A. The second part of Condition A follows trivially from (\ref{e2}) and a calculation similar to the above.

\vspace{0.05in}

\noindent{\bf Condition C:}  
For proving Condition $C$ we use the following theorem from \cite{CG08}:

\begin{theorem} (Theorem 1 in \cite{CG08}) Let $\bar{m}(n)$ be such that
\begin{equation}
 n\log n \ll \bar{m}(n) \ll 2^n (\log n)^{-1},
\end{equation}
and let $A_n$ be a sequence of percolation clouds on $\V_n$ with densities $\bar{m}(n)^{-1}$. Then, for all $a>0$,
\begin{equation}
 \lim_{n\to\infty}\max_{x\in \vn}\left|P_x\Big(H_n(A_n\setminus \{x\})\geq a\bar{m}(n)\Big)-\exp(-a)\right|=0.
\end{equation}
\end{theorem}

In our case, Condition $C$ is equivalent to the result of the above theorem when $\bar{m}(n)=(\rho\rhon)^{-1}$, $\rho>0$. We have in (\ref{scales}) $\rhon^{-1}=\an\beta_n\sqrt{2\pi n}\exp(\alpha_n^2\beta_n^2 n/2)$ and, thus, (\ref{assum0}) and (\ref{assum}) yield $n\log n \ll (\rho\rhon)^{-1}\ll 2^n (\log n)^{-1}$. Hence, we can apply the theorem and verify Condition $C$.

\vspace{0.05in}

\noindent{\bf Condition B:} Recall the notation $\xi_n=T\rn$. It is trivial that for any $k_1,k_2$ with $k_1\leq k_2$, and for all $x,y\in\V_n$, $G^n_{k_1}(x,y)\leq G^n_{k_2}(x,y)$. Also, for any $k\in\N$ and $x\in\V_n$, $G^n_k(x,x)\geq 1$ by definition. Hence, using part (a) of Lemma 3.9 in \cite{BC08} we conclude that
\begin{equation}\label{cbass}
\limsup_{n\to\infty} G_{\xi_n\log \xi_n}^n(0,0)=1.
\end{equation}
Now we claim that uniformly for $x\in \tem$ that
\begin{equation}\label{cellop}
\lim_{n\to\infty}G_{\tem\setminus \{x\}}^n(x,x)=1,
\end{equation}
which is enough to check Condition B with $\fn=1$ and $\mathcal{K}_G=1$. Recalling that $H_n'(x)=\min\{i\geq 1: Y_n(i)=x\}$ we have
\begin{equation}
G_{\xi_n\log \xi_n}^n(0,0)=\big(P_x(H_n(\tem\setminus\{x\})<H_n'(x))\big)^{-1}.
\end{equation}
We have
\begin{equation}
P_x(H_n(\tem\setminus\{x\})\geq H_n'(x))\leq P_x(H_n'(x)\leq \xi_n\log \xi_n)+P_x(H_n(\tem\setminus \{x\})\geq \xi_n\log\xi_n).
\end{equation}
The first term on the right hand side of the above display is obviously independent of $x$ and converges to 0 by (\ref{cbass}); using Condition C, that is, the fact that uniformly for $x\in\tem$, $H_n(\tem\setminus\{x\})/\xi_n$ is asymptotically an exponential random variables, we can conclude that the second term also vanishes uniformly for $x\in\V_n$. Hence, we have proved (\ref{cellop}). 

\vspace{0.05in}

\noindent{\bf Condition D:} By Lemma 3.9 in \cite{BC08} we have for all $x\not= 0$, for all $n$ large enough
\begin{equation}\label{pcello}
G_{\xi_n}^n(0,x)\leq C/n
\end{equation}
for some positive constant $C>0$. By (\ref{assum}) we can choose a sequence $\lambda_n$ such that
\begin{equation}
\lambda_n\ll n,\;\;\;\lambda_n\ll \ann\beta_n^2 n
\end{equation}
and $\sum_n\exp(-\lambda_n)<\infty$ (recall that $\fn=1$). For such $\lambda_n$, using (\ref{pcello}) we have $\lambda_n G_{\xi_n(0,x)}\ll 1$ uniformly for all $x\in\V_n\setminus\{x\}$. Combining this with (\ref{cbass}) we get that for some positive constant $K$, for all $n$ large enough
\begin{equation}
\sum_{x\in \V_n}(e^{\lambda_n G_{\xi_n}^n(0,x)}-1)\leq e^{2\lambda_n}+\sum_{x\in \V_n\setminus\{x\}}(e^{\lambda_n G_{\xi_n}^n(0,x)}-1)\leq e^{2\lambda_n}+\lambda_n T\rn\leq K T\lambda_n\rn
\end{equation}
where in the last step we used that $\lambda_n\ll\ann\beta_n^2 n$. Hence, we have checked the first part of Condition D. The second part of Condition D is trivial since $\fn=1$.

\vspace{0.05in}

\noindent{\bf Condition 1:} Let $t'_n$ be a deterministic sequence of times satisfying (\ref{tnsqn}) and $A_n(\delta)$ be defined as (\ref{ansqn}). We check that Condition 1 is satisfies uniformly for $0<a<b$, conditioned on $m_n(q_n(j_n))=a^{1/\an}\tn$ and $m_n(q_n(j_n+1))=b^{1/\an}\tn$. For $\ep>0$ small enough so that $a+2\ep<b-2\ep$ we define the sequence of events:
\begin{equation}
 B_n:=\left\{\tau_{Y_n(i)}\leq (b-2\ep)^{1/\an} \tn:\;\forall i=k_n(j_n),\dots,k_n(j_n+1)-1\right\}.
\end{equation}
By Condition C and Proposition \ref{asympindp} we can choose $\ep$ small enough so that $\btau$ a.s. for $n$ large enough we have
\begin{equation}
\pp(B_n|\btau)\geq 1-\delta/4.
\end{equation}
By Proposition \ref{keyp}, $\btau$ a.s. for $n$ large enough
\begin{equation}
\pp\left(S_n(k_n(j_n+1))-S_n(k_n(j_n))\geq (b-\ep/2)^{1/\an}\tn-(a+\ep)^{1/\an}\tn\big|\btau\right)\geq 1-\delta/4.
\end{equation}
Note that $(b-\ep/2)^{1/\an}\tn-(a+\ep)^{1/\an}\tn \gg (b-2\ep)^{1/\an}\tn$ as $n\to\infty$. Also, by Proposition \ref{shallowt}, on $I_n$, the contribution from traps between $k_n(j_n)$ and $k_n(j_n+1)$ is smaller than $(b-\ep)^{1/\an}\tn$ with a probability larger than $1-\delta/4$. Hence, conditioned on $I_n$, $A_n(\delta)$ and the sequence of events inside the probability term in the last display, the only way $X_n(t'_n)\not= V_n(j_n)$ is if the random walk comes back to $V_n(j_n)$ before visiting $V_n(j_n+1)$. However, by (\ref{cbass}) this probability goes to 0. Thus, Condition 1 is satisfied.

\vspace{0.05in}

\noindent{\bf Condition 2:} By Condition C, the time for $Y_n$ to visit a deep trap divided by $\rn$ is approximately exponentially distributed. Hence, by (\ref{cbass}), we have that the probability of $Y_n$ revisits a deep trap converges to 0. Thus, Condition 2 follows.

\end{proof} 

\end{section}

\bibliographystyle{plain}
\bibliography{DynamicsSpinGlasses}
\end{document}